\documentclass[12pt,multicol,english]{amsart}
\usepackage{multicol}
\usepackage[T1]{fontenc}
\usepackage[utf8]{inputenc}
\usepackage[a4paper]{geometry}
\usepackage{amssymb,amsmath}
\usepackage{version}
\usepackage{mathtools} 
\usepackage{enumerate}
\usepackage{dsfont}
\usepackage{mathdots}
\makeatletter

\usepackage{latexsym}
\usepackage{color}

\def\subsubsection{\@startsection{subsubsection}{3}%
	\z@{.5\linespacing\@plus.7\linespacing}{-.5em}%
	{\normalfont\bfseries}}
\newtheorem{thm}{Theorem}[section]
\newtheorem{lemme}[thm]{Lemma}

\newtheorem{prop}[thm]{Proposition}

\newtheorem{ex}[thm]{Example}
\newtheorem{coro}[thm]{Corollary}
\newtheorem{defi}[thm]{Definition}
\newtheorem{rem}[thm]{Remark}

\newtheorem{fact}{Fact}

\newtheorem*{thm*}{Theorem}

\def\D{{\mathbb{D}}}

\def\R{{\mathbb{R}}}
\def\C{{\mathbb{C}}}

\def\N{{\mathbb{N}}}

\def\T{{\mathbb{T}}}
\def\UU{{\mathcal{U}}}

\def\FF{{\mathcal{F}}}

\def\GG{{\mathcal{G}}}

\def\CC{{\mathcal{C}}}
\def\MM{{\mathcal{M}}}
\def\WW{{\mathcal{W}}}

\numberwithin{equation}{section} 
\numberwithin{figure}{section} 
\numberwithin{table}{section} 

\usepackage[bookmarksopen, naturalnames]{hyperref}

\makeatother

\usepackage{babel}

\begin{document}
	\title{Universal sequences of composition operators}
	\author{S. Charpentier, A. Mouze}
	
	\address{St\'ephane Charpentier, Institut de Math\'ematiques, UMR 7373, Aix-Marseille
	Universite, 39 rue F. Joliot Curie, 13453 Marseille Cedex 13, France}
	\email{stephane.charpentier.1@univ-amu.fr}
	
	\address{Augustin Mouze, Laboratoire Paul Painlev\'e, UMR 8524, Cit\'e Scientifique, 59650 Villeneuve
	d'Ascq, France, Current address: \'Ecole Centrale de Lille, Cit\'e Scientifique, CS20048, 59651
	Villeneuve d'Ascq cedex, France}
	\email{augustin.mouze@univ-lille.fr}
	
	\thanks{The authors are supported by the grant ANR-17-CE40-0021 of the French National Research Agency ANR (project Front).}
	
	\begin{abstract}Let $G$ and $\Omega$ be two planar domains. We give necessary and sufficient conditions on a sequence $(\phi_n)$ of eventually injective holomorphic mappings from $G$ to $\Omega$ for the existence of a function $f\in H(\Omega)$ whose orbit under the composition by $(\phi_n)$ is dense in $H(G)$. This extends a result of the same nature obtained by Grosse-Erdmann and Mortini when $G=\Omega$. An interconnexion between the topological properties of $G$ and $\Omega$ appears. Further, in order to exhibit in a natural way holomorphic functions with wild boundary behaviour on planar domains, we study a certain type of universality for sequences of continuous mappings from a union of Jordan curves to a domain.
	\end{abstract}
	
	\maketitle

\section{Introduction}

Let $\Omega$ be a domain ({\it i.e.} a connected open set) in $\C$ and let $H(\Omega)$ denote the Fréchet space of holomorphic functions on $\Omega$, endowed with the locally uniform topology. Given two domains $G$ and $\Omega$, we say that a sequence $\Phi=(\phi_n)_n$ of holomorphic functions from $G$ to $\Omega$ is \textit{universal} if there exists a function $f\in H(\Omega)$ such that the set
\[
\{f\circ \phi_n:\,n\in \N\}
\]
is dense in $H(G)$. If such function $f$ exists, we call it a $(\Phi,G,\Omega)$-universal function.

When $G=\Omega$, universal sequences of holomorphic selfmaps of $\Omega$ have been subject of many researches. In this setting, the first examples were provided for $\Omega=\C$ by Birkhoff \cite{Birkhoff} with sequences $(\phi_n)_n$ of the form $\phi_n(z)=z+n$. In 1995, Bernal and Montès \cite{BerMon} have characterised the sequences of automorphisms admitting a universal function, provided $\Omega$ is not conformally equivalent to the punctured plane $\C^*$. Later, Grosse-Erdmann and Mortini obtained a complete description of those sequences $\Phi=(\phi_n)_n$ of {\it eventually injective} holomorphic functions from a domain $\Omega$ into itself that are universal \cite[Theorem 3.19]{G-Erdmann-Mortini}. In this setting, it appears that the existence of such universal sequences depends on the geometry of $\Omega$.
For instance, while every simply connected domain supports a universal sequence of holomorphic automorphisms \cite{BerMon}, a finitely connected domain supports universal sequences of injective holomorphic mappings if and only if it is simply connected. Without the injectivity assumption, the problem becomes more involved and pathological examples of universal sequences of non-injective holomorphic selfmaps of a domain (possibly finitely connected but not simply connected) can be exhibited (see \cite[Remark 3.7]{G-Erdmann-Mortini}). Roughly speaking, the reason for the injective case to be so restrictive is that injective holomorphic mappings preserve the most important topological property of compact sets in complex approximation, namely the number of connected components of their complement.

\medskip

We refer to the references listed in \cite{G-Erdmann-Mortini} for a complete overview of the subject before 2009. The topic has been extensively continued in various directions later on, see for \emph{e.g.} \cite{Ber,BerCalJungPra,BerCalPra,Bes1,Bes2,Jungphd,Jung,JungMul,Meyr,Zajac}.

\medskip

Despite the large amount of contributions, it seems that the problem of the existence of holomorphic functions universal for sequences of composition operators have never been studied when the domains $G$ and $\Omega$ differ. Let us however mention a very recent paper by Meyrath \cite{Meyr} in which the author gives necessary and sufficient conditions on a sequence $(\phi_n)_n$ of holomorphic functions from one domain $G$ to an other domain $\Omega$ that guarantee the existence of \emph{meromorphic} functions in $H(\Omega)$ such that the set $\{f\circ \phi_n:\,n\in \N\}$ is dense in the Fréchet space of meromorphic functions on $G$. Rather naturally, it appears that no topological conditions are needed on $G$ and $\Omega$. The reason is practically that allowing (universal) approximation by meromorphic functions allows one to "create" holes anywhere in the domain $\Omega$. This is also why requiring holomorphic approximation becomes more involved and makes the above problem natural and interesting, in particular from a topological point of view.

Now let us make a simple observation. Clearly, the existence of universal sequences of holomorphic functions is in some sense conformally invariant: if $G'$ is conformally equivalent to $G$, and $\Omega'$ conformally equivalent to $\Omega$, then a sequence $(\phi_n)_n$ of holomorphic functions from $G$ to $\Omega$ is universal whenever the sequence $(\varphi \circ \phi_n \circ \psi)$ from $G'$ to $\Omega'$ is itself universal, where $\varphi$ and $\psi$ are conformal maps. In particular, the case where $G$ and $\Omega$ are conformally equivalent reduces to $G=\Omega$. Thus, investigating the case where $G$ and $\Omega$ are not conformally equivalent is of particular interest.

In the first part of this paper, for any domains $G$ and $\Omega$, we obtain a complete characterisation of the sequences $\Phi=(\phi_n)_n$ of eventually injective holomorphic mappings from $G$ to $\Omega$ for which $(\Phi,G,\Omega)$-universal functions exist. More precisely:
\begin{thm*}Let $G$ and $\Omega$ be two domains and let $\Phi=(\phi_n)_n$ be a sequence of eventually injective holomorphic mappings from $G$ to $\Omega$. A necessary and sufficient condition that guarantees the existence of $(\Phi,G,\Omega)$-universal functions is that for any $G$-convex compact subset $K$ of $G$, any $\Omega$-convex compact subset $L$ of $\Omega$ and any $N\in \N$, there exists $n\geq N$ such that $\phi_n(K)\cap L=\emptyset$ and $\phi_n(K)\cup L$ is $\Omega$-convex.
\end{thm*}
We recall that a compact subset $K$ of a domain $G$ is $G$-convex if any non-empty bounded connected component of $\C\setminus K$ contains a point of $\C\setminus G$. It will be deduced from this theorem that if there exist universal eventually injective  sequences $\Phi$ from $G$ to $\Omega$, then $G$ is simply connected or $\Omega$ is infinitely connected. Let us recall that if $G=\Omega$, then $G$ supports universal eventually injective sequences of holomorphic selfmaps only if it is simply connected or infinitely connected \cite{G-Erdmann-Mortini}. Illustrating the difference with the case where $G$ is equal to $\Omega$, we will see that for any bounded domain $G$, there exist (non-trivial) domains $\Omega$ for which there exist universal injective sequences $\Phi$ from $G$ to $\Omega$. Furthermore, if $G$ is simply connected or has at least two holes, the $\Omega$-convexity of $\phi_n(K)\cup L$ can be replaced with the $\Omega$-convexity of $\phi_n(K)$. In particular, our results cover that of \cite[Section 3]{G-Erdmann-Mortini}. When $G$ is doubly connected and $\Omega$ infinitely connected, we will see that the condition $\phi_n(K)\cup L$ cannot be replaced with the $\Omega$-convexity of $\phi_n(K)$.

\medskip

In a second part of the paper, motivated by exhibiting holomorphic functions with wild boundary behaviour, we will study another kind of compositional universality. To start with, let us consider a very simple setting. We denote by $\D=\{z\in \C:\, |z|<1\}$ the unit disc of the complex plane, and by $\T$ its boundary. Let $(r_n)_n$ be a sequence in $[0,1)$ converging to $1$, and let $\phi_n : \T \to \D$ defined by $\phi_n(z)=r_n z$. As a particular case of a result stated in \cite{Charp}, there exists a function $f \in H(\D)$ such that the set $\{f\circ \phi_n:\,n\in \N\}$ is dense in the Banach space $\CC(K)$ of all continuous functions on $K$ (endowed with the sup-norm), for any subset $K$ of $\T$ different from $\T$. Such functions enjoy a very singular behaviour near the boundary of $\T$. In general, this type of boundary behaviour is (strictly) wilder than having a dense cluster set along any continuous path to the boundary - another example of erratic boundary behaviour which was considered in several papers, see \cite{Prado} and the references therein. 
Note that the sequence $(\phi_n)_n$ can also be seen as a sequence of holomorphic selfmaps of $\D$ and that in this case, it is obviously not universal in the sense considered in the first part of this introduction. Our aim is to define another setting for compositional universality making the last result a natural example of which. Let $G$ be a closed subset of $\C$ with empty interior, let $\MM$ be a family of compact subsets of $G$ with empty interior, and let $\Omega$ be a domain in $\C$. We say that a sequence $\Phi=(\phi_n)_n$ of continuous functions from $G$ to $\Omega$ is \emph{$\MM$-universal} if there exists a function $f\in H(\Omega)$ such that, for any $K\in \MM$, the set $\{f\circ \phi_n:\,n\in \N\}$ is dense in $\CC(K)$. Such a function $f$ will be called a $(\Phi,\MM)$-universal function. In order to simplify the notation, we omit the dependence in $G$ and $\Omega$
. This definition may look a bit articifial, so let us comment on some justifications. First of all, in order not to fall in the definition of universality considered in the beginning, it is natural to consider that the functions $\phi_n$'s are only assumed to be continuous. Now, because we still want our universal function $f$ to be holomorphic, we will have to use complex polynomial approximation to approximate \textit{continuous} functions on some sets of the form $\phi_n(K)$, where $K$ is a compact subset of $G$. This can often be impossible if $K$ does not have empty interior.

The second part of this paper is thus devoted to give conditions, necessary or sufficient, for some sequence $\Phi$ to be universal in this new setting. We will focus especially on the more suitable - and quite illustrative - case where $G$ is a union of Jordan curves. As desired, this will allow us to exhibit in a simple way functions holomorphic on rather general domains that enjoy a wild boundary behaviour.





\medskip


\section{Preliminaries}

In the whole paper, $\N$ denotes the set $\{0,1,2,3,\ldots\}$ of all non-negative integers. A sequence of general terms $u_n$, $n\in \N$, will be denoted by $(u_n)_n$.

\medskip

Let us introduce some notations and terminology related to topological notions.

\begin{itemize}

\item If $E$ is a subset of $\C$, we denote by $\partial E$ its boundary and by $\text{int}(E)$ its interior.

\item If $\Omega$ is a domain in $\C$, we call an \textit{exhaustion} of $\Omega$ a sequence $(K_n)_n$ of non-empty compact subsets of $\Omega$ which satisfy the following: $K_n\subset \text{int}(K_{n+1})$ and for any compact set $K\in \Omega$, there exists $n\in \N$ such that $K\subset K_n$. We recall that any domain in $\C$ admits an exhaustion.
	
\item Let $\Omega$ be a domain and $\Phi=(\phi_n)_n$ a sequence of holomorphic functions on $\Omega$. We say that $\Phi$ is \textit{eventually injective} if for any compact set $K$, there exists $N\in \N$ such that $\phi_n$ is injective on $K$ (or equivalently on a neighbourhood of $K$) for any $n\geq N$.

\item A \textit{Jordan curve} is the image of the unit circle by an injective continuous map. A \textit{Jordan arc} is the image of $[0,1]$ by an injective continuous map (note that with our definition a Jordan curve is not a Jordan arc, and \textit{vice versa}).

\item We say that a compact subset $K$ of $\C$ is \textit{regular} if $\partial K$ is a union of Jordan curves.

\item If $E$ is a subset of $\C$, we call a \textit{hole} of $E$ a non-empty bounded connected component of the complement $\C\setminus E$ of $E$.

\item A domain is said to be \textit{simply connected} if it has no hole, \textit{doubly connected }if it has one hole, \textit{finitely connected} if it has a positive finite number of holes, and \textit{infinitely connected} if it has infinitely many holes.

\item If $K$ is a compact subset of $\C$, we denote by $\widehat{K}$ its \textit{polynomial hull}, that is the union of $K$ and of all its holes.

\item If $K$ is a regular compact subset of $\C$, we call \textit{outer boundary} of $K$ the boundary (in $\C$) of the unbounded component of its complement, and by \textit{inner boundary} the boundary of the union of its holes.

\item A compact subset $K$ of an open set $\Omega$ in $\C$ is said to be $\Omega$-convex if every hole of $K$ contains a point of $\C\setminus \Omega$.

Note that if a compact subset $K$ of an open set $\Omega$ has a hole that contains a point of $\C\setminus\Omega$, then this hole contains the hole of $\Omega$ containing this point.

\item Let $\Omega$ be a domain. We say that a non-empty compact set $K$ in $\Omega$ is \emph{$\Omega$-connected} if it has no hole whenever $\Omega$ is simply connected, and does have a hole if $\Omega$ is not simply connected.

\item If $\Omega$ is a domain, the notation $\MM(\Omega)$ will stand for the set of all compact subsets of $\Omega$ which are regular, connected, $\Omega$-connected and $\Omega$-convex.

\end{itemize}

Runge and Mergelyan Theorems are at the core of the construction of universal holomorphic functions. Let us recall them.

\begin{thm}[Runge Theorem, see \cite{Rudin}]Every function holomorphic in a neighbourhood of a compact set $K\subset \C$ can be uniformly approximated on $K$ by rational functions with poles off $K$.
\end{thm}

We recall that any $\Omega$-convex compact subset of a domain has at most finitely many holes (see for example \cite[Lemma 3.10]{G-Erdmann-Mortini}). Thus we deduce from Runge Theorem and \cite[Theorem 4, p119]{Gaier} the following version of Mergelyan Theorem.

\begin{thm}[Mergelyan Theorem]\label{Mergelyan}Let $\Omega$ be a domain and $K$ a compact $\Omega$-convex set. Every continuous function on $K$, holomorphic in its interior, can be approximated uniformly on $K$ by rational functions with poles off $\Omega$.
\end{thm}

To finish, the following lemma can be seen as a very special case of Kallin's lemma (see, for e.g., \cite[Theorem 1.6.19]{Stout}).

\begin{lemme}\label{lem-union-omega-convex}Let $\Omega$ be a domain and let $K,L$ be two $\Omega$-convex compact sets. If $K\subset \C\setminus \widehat{L}$ and $L\subset \C\setminus \widehat{K}$, then $K\cup L$ is $\Omega$-convex.
\end{lemme}

\section{The case where $G$ is a domain}\label{section-domain}
In this section, we consider two domains $G$ and $\Omega$, and a sequence $\Phi=(\phi_n)_n$ of holomorphic mappings from $G$ to $\Omega$. We will focus on the following definition.

\begin{defi}\label{defi-ori-domain}
We say that $f\in H(\Omega)$ is $(\Phi,G,\Omega)$-\textit{universal} if the set $\{f\circ \phi_n:\,n\in \N\}$ is dense in $H(G)$.
\end{defi}

This definition coincides with \cite[Definition 1.1 (a)]{G-Erdmann-Mortini} in the case where $G=\Omega$. The Birkhoff universality theorem provides a criterion for compositional universality which will be useful in the sequel. The proof is standard and left to the reader.

\begin{prop}\label{Birk}
There exists a $(\Phi,G,\Omega)$-universal function if and only if for any $\varepsilon>0$, any compact set $K$ in $G$, any compact set $L$ in $\Omega$, any function $g\in H(G)$ and any function $h\in H(\Omega)$, there exists $f\in H(\Omega)$ and $n\in \N$ such that
\[
\sup_{z\in K}|f\circ \phi_n(z) - g(z)|\leq \varepsilon \quad \text{and}\quad \sup_{z\in L}|f(z) - h(z)|\leq \varepsilon.
\]
\end{prop}



The following lemma is a direct consequence of the beginning of the proof of \cite[Theorem 3.12]{G-Erdmann-Mortini}.

\begin{lemme}\label{exhaustion-convenient}Let $\Omega\subset \C$ be a domain. There exists an exhaustion $(L_n)_n$ of $\Omega$ by compact sets in $\MM(\Omega)$.
\end{lemme}

We immediately deduce from Proposition \ref{Birk} and the previous lemma the following fact.

\begin{fact}\label{fact-Birk-regular}
There exists a $(\Phi,G,\Omega)$-universal function if and only if for any $\varepsilon>0$, any compact set $K\in \MM(G)$, any compact set $L$ in $\MM(\Omega)$, any function $g\in H(G)$ and any function $h\in H(\Omega)$, there exists $f\in H(\Omega)$ and $n\in \N$ such that
\[
\sup_{z\in K}|f\circ \phi_n(z) - g(z)|\leq \varepsilon \quad \text{and}\quad \sup_{z\in L}|f(z) - h(z)|\leq \varepsilon.
\]
\end{fact}

\medskip

The main theorem of this section states as follows.

\begin{thm}\label{thm-nec-inj-one-infinity-holes}
Let us assume that $\Phi =(\phi_n)_n$ is eventually injective. There exists a $(\Phi,G,\Omega)$-universal function if and only if for every compact set $K\in \MM(G)$, every compact set $L\in \MM(\Omega)$, and every $N\in \N$, there exists $n\geq N$ such that $\phi_n(K)\cap L = \emptyset$ and $\phi_n(K)\cup L$ is $\Omega$-convex.
	
Moreover, if $G$ is not simply connected and $\Omega$ is not infinitely connected, then there do not exist $(\Phi,G,\Omega)$-universal functions.
\end{thm}

We need a lemma.

\begin{lemme}\label{lemme-one-infinity-holes}Let 
$L$ and $L'$ be two disjoint connected $\Omega$-convex compact sets. Then $L\cup L'$ is $\Omega$-convex if and only if one of the following conditions holds:
\begin{enumerate}
	\item $L\subset \C\setminus \widehat{L'}$ and $L'\subset \C\setminus \widehat{L}$;
	\item there exists a hole $O$ of $L$ such that $L'\subset O$ and $O\setminus \widehat{L'}$ contains a hole of $\Omega$;
	\item there exists a hole $O'$ of $L'$ such that $L\subset O'$ and $O'\setminus \widehat{L}$ contains a hole of $\Omega$.
\end{enumerate}
\end{lemme}

\begin{proof}Let assume that (1) holds. Then the set of holes of $L\cup L'$ is the union of the sets of all holes of $L$ and $L'$. Since both $L$ and $L'$ are $\Omega$-convex, each hole of $L\cup L'$ contains a hole of $\Omega$, which means that $L\cup L'$ is $\Omega$-convex (the same remains true if $L$ or $L'$ does not have holes).
	
Since $L$ and $L'$ are disjoint and connected, if (1) does not hold then $\widehat{L}$ is contained in a hole of $L'$, or $\widehat{L'}$ is contained in a hole of $L$. The two situations being symmetrical, let us assume that $\widehat{L}$ is contained in a hole of $L'$, that we call $O'$.  Then, again by connectivity, the set of holes of $L\cup L'$ consists in the union of the holes of $L$ and $L'$, except $O'$, and the (connected) set $O'\setminus \widehat{L}$. Since the holes of $L$ and $L'$ each contains a hole of $\Omega$, the set $L\cup L'$ is thus $\Omega$-convex if and only if $O'\setminus \widehat{L}$ contains a hole of $\Omega$.
\end{proof}

Let us turn to the proof of Theorem \ref{thm-nec-inj-one-infinity-holes}.

\begin{proof}[Proof of Theorem \ref{thm-nec-inj-one-infinity-holes}]The if part is straightforward. It simply consists in applying Fact \ref{fact-Birk-regular}, by using Runge's theorem  on $\phi_{n}(K)\cup L$ to the function
\[
\kappa(z)=\left\{
\begin{array}{ll}h(z) & \text{if }z\in L\\
g\circ \phi_n^{-1}(z) & \text{if }z\in \phi_{n}(K).
\end{array}
\right.
\]

\medskip

Let us start with the necessity condition. Let us fix $N\in\N$, $K\subset G$ and $L\subset \Omega$ as in the statement. Note that $K$ and $L$ have only finitely many holes \cite[Lemma 3.10]{G-Erdmann-Mortini}.

We first assume that $G$ is not simply connected. Then since $K$ is $\Omega$-connected, it has $p\geq 1$ holes. Let us denote by $\gamma_0$ the outer boundary of $K$ and by $\gamma_i$, $1\leq i\leq p$, the connected components of the inner boundary of $K$. We assume that $\gamma_0$ is positively oriented and that $\gamma_i$, $1\leq i\leq p$, is negatively oriented. Consider $p$ Jordan arcs $I_i$, $1\leq i\leq p$, in $G$ such that one extremity of $I_i$ is in $\gamma_0$ and the other is in $\gamma_i$. Since $K$ is connected, we may and shall assume that the $I_i$'s are contained in $K$ and pairwise disjoint. Let us fix $b\in K\setminus \cup_{i=1}^pI_i$ and $\lambda _i$, $1\leq i\leq p$, in the hole of $K$ bounded by $\gamma_i$. Since $K$ is $G$-convex, we may and shall assume that each $\lambda_i$ lies in some hole of $G$. We now consider the functions
	\[
	g_m(z)=m\frac{(z-b)^{p+1}}{\prod_{i=1}^p(z-\lambda_i)},\quad m\in \N.
	\]
They all belong to $H(\Omega)$. By the argument principle, we have for every $m\in \N$,
\begin{equation}\label{eq1-one-infinity}
\frac{1}{2i\pi}\int_{\gamma_0}\frac{g_m'(z)}{g_m(z)}dz = \frac{1}{2i\pi}\int_{\gamma_i}\frac{g_m'(z)}{g_m(z)}dz=1,\quad i=1\ldots p.
\end{equation}
Let $f\in H(\Omega)$ be a $(\Phi,G,\Omega)$-universal function. For every $m\in \N$, there exists a sequence $(n_k)_k$ such that $f\circ \phi_{n_k} \to g_m $ in $H(G)$ and $\displaystyle{\frac{(f\circ \phi_{n_k})'}{f\circ \phi_{n_k}} \to \frac{g_m'}{g_m}}$ in $H(G\setminus \{b\})$. Then by definition of $g_m$ and by \eqref{eq1-one-infinity}, and upon choosing $m$ large enough, there exists an integer $m_1\geq N$ such that $\phi_{m_1}$ is injective in a neighbourhood of $K$,
\begin{equation}\label{eq2-one-infinity}
\min _{z\in \partial K \cup \bigcup_{i=1}^p I_i}|f\circ \phi_{m_1}(z)| > \max_{z\in L}|f(z)|
\end{equation}
and
\begin{equation}\label{eq3-one-infinity}
\frac{1}{2i\pi}\int_{\phi_{m_1}(\gamma_0)}\frac{f'(z)}{f(z)}dz = \frac{1}{2i\pi}\int_{\phi_{m_1}(\gamma_i)}\frac{f'(z)}{f(z)}dz=1, \quad i=1,\ldots ,p.
\end{equation}
Equation \eqref{eq3-one-infinity} implies that $\phi_{m_1}(K)$ is $\Omega$-convex. Indeed, assume by contradiction that it is not. Since $\phi_{m_1}$ is injective on $K$, $\phi_{m_1}(K)$ has $p$ holes, and $f$ is holomorphic on one of them. The boundary of this hole is $\phi_{m_1}(\gamma_i)$ for some $i=0,\ldots,p$ and is negatively oriented, since injective holomorphic mappings send boundaries to boundaries and preserve orientation. It follows from the argument principle and \eqref{eq3-one-infinity} that $f$ has $-1$ zero in this hole, which is impossible. A straight consequence of this fact is that $\Omega$ cannot be simply connected (in other words, there do not exist $(\Phi,G,\Omega)$-universal functions whenever $G$ is not simply connected and $\Omega$ is simply connected).

To finish with the case where $G$ is not simply connected, we thus assume from now on that $\Omega$ is not simply connected. In particular, since $L\in \MM(\Omega)$, $L$ has at least one hole. Now, Equation \eqref{eq2-one-infinity} clearly implies that $\phi_{m_1}(\partial K\cup \bigcup_{i=1}^p I_i) \cap L=\emptyset$. Since again injective holomorphic mappings send boundaries to boundaries, and since $L$ is connected, has at least one hole and is $\Omega$-convex, we infer that $\phi_{m_1}(K) \cap L=\emptyset$.

If $\phi_{m_1}(K) \cup L$ is $\Omega$-convex, then we are done. Assume that it is not the case. Since $L$ and $\phi_{m_1}(K)$ are connected, then by Lemma \ref{lemme-one-infinity-holes}, either $L$ is contained in a hole - let say $O_1$ - of $\phi_{m_1}(K)$, or $\phi_{m_1}(K)$ is contained in a hole - let say $O'_1$ - of $L$, and $O_1\setminus \widehat{L}$ and $O'_1\setminus \widehat{\phi_{m_1}(K)}$ contain no point of the complement of $\Omega$. Applying exactly the same argument as above with $L$ replaced with the compact set
\[
L_1:=L\cup \phi_{m_1}(K) \cup (O_1\setminus \widehat{L}) \cup (O'_1\setminus \widehat{\phi_{m_1}(K)}),
\]
we get the existence of an integer $m_2>m_1$ such that $\phi_{m_2}$ is injective on a neighbourhood of $K$, $\phi_{m_2}(K)$ is $\Omega$-convex, $\phi_{m_2}(K)\cap L_1=\emptyset$ and \eqref{eq3-one-infinity} holds with $m_2$ instead of $m_1$.

Since $L\subset L_1$ and the holes of $L_1$ are contained in that of $L$ by assumption, if $\phi_{m_2}(K) \cup L_1$ is $\Omega$-convex, then $\phi_{m_2}(K) \cup L$ is also $\Omega$-convex. So if this is the case, then we are done. If not, by Lemma 
\ref{lemme-one-infinity-holes}, then either $\phi_{m_2}(K)$ is contained in a hole $O'_2$ of $L_1$, or $L_1$ is contained in a hole $O_2$ of $\phi_{m_2}(K)$, and $O_2\setminus \widehat{L_1}$ and $O'_2\setminus \widehat{\phi_{m_2}(K)}$ contain no point of the complement of $\Omega$. So we can again reproduce the same step with $L_1$ replaced with
\[
L_2:=L_1\cup \phi_{m_2}(K) \cup (O_2\setminus \widehat{L}) \cup (O'_2\setminus \widehat{\phi_{m_2}(K)}).
\]
Doing so, since the complement of $L$ has only finitely many connected components, we can exhibit, after finitely many steps of such construction, a pair of integers $(m_{k_1},m_{k_2})$ with $m_{k_1}<m_{k_2}$ such that, for $i=1,2$, $\phi_{m_{k_i}}$ is injective on a neighbourhood of $K$, $\phi_{m_{k_i}}(K)$ is $\Omega$-convex, $\phi_{m_{k_i}}(K)\cap L =\emptyset$, \eqref{eq3-one-infinity} holds with $m_{k_i}$ instead of $m_1$, and such that one of the following holds:
\begin{enumerate}[(a)]
	\item $\phi_{m_{k_1}}(K)\cup L$ or $\phi_{m_{k_2}}(K)\cup L$ is $\Omega$-convex;
	\item $\phi_{m_{k_1}}(K)\cup L$ and $\phi_{m_{k_2}}(K)\cup L$ are not $\Omega$-convex, and $\phi_{m_{k_2}}(K)$ and $\phi_{m_{k_1}}(K)$ are contained in a hole of the other.
\end{enumerate}
In Case (a) we have the conclusion of the theorem. Let us then focus on Case (b) and let us assume that $\phi_{m_{k_2}}(K)$ is contained in a hole $O_{k_1}$ of $\phi_{m_{k_1}}(K)$. The proof is similar in the other situation. Then, by construction, $\phi_{m_{k_1}}(K)$ and $\phi_{m_{k_2}}(K)$ are contained in a same hole of $L$. Thus, since $\phi_{m_{k_2}}(K)\cup L$ is not $\Omega$-convex, the set $O_{k_1}\setminus \phi_{m_{k_2}}(K)$ contains no point of the complement of $\Omega$. In particular, $f$ is holomorphic in a neighbourhood of $O_{k_1}\setminus \phi_{m_{k_2}}(K)$. Using one more time that injective holomorphic mappings preserve orientation, we get that the outer boundary $\Gamma_1$ of $O_{k_1}\setminus \phi_{m_{k_2}}(K)$ is $\phi_{m_{k_1}}(\gamma_i)$ for some $i=0,\ldots ,p$ and is negatively oriented, while its inner boudary $\Gamma_0$ is the outer boundary of $\phi_{m_{k_2}}(K)$, which is $\phi_{m_{k_2}}(\gamma_i)$ for some $i=0,\ldots ,p$ and is positively oriented.

It then follows that the quantity $-\int_{\Gamma_0+\Gamma_1}\frac{f'(z)}{f(z)}dz$ is equal to the number of zeroes of $f$ in $O_{k_1}\setminus \phi_{m_{k_2}}(K)$. Yet this quantity is also equal to $-2$ by \eqref{eq3-one-infinity} applied with $m_{k_i}$, $i=1,2$, instead of $m_1$. This shows that Case (a) occurs.

\medskip

Before turning to the case where $G$ is simply connected, let us prove the last assertion of the theorem. To do so, we thus assume that $\Omega$ is not infinitely connected. It is enough to see that, in the previous construction, Case (a) never occurs for some choice of the compact set $L$. Let us denote by $1\leq q < +\infty$ the number of holes of $\Omega$, and let us assume that $L$ has exactly $q$ holes. Then, since $\phi_{m_i}(K) \cap L=\emptyset$ for $i=1,\ldots, k_2$, the set $\phi_{m_i}(K) \cup L$ has at least $p+q > q$ holes and one of this hole does not contain any point of the complement of $\Omega$. Hence, for any $i=1,\ldots, k_2$, $\phi_{m_i}(K) \cup L$ is not $\Omega$-convex and Case (b) occurs. As we have seen, it leads to a contradiction.

\medskip

To conclude the necessity part, let us consider the easy case where $G$ is simply connected. Let $K$ and $L$ as in the statement. In particular, $K$ has connected complement. Let $f\in H(\Omega)$ be a $(\Phi,G,\Omega)$-universal function and consider $M\in \R$ such that $\varepsilon:=M - \sup_{z\in L}|f(z)|>0$. By universality, there exists an integer $n\geq N$ such that $\phi_n$ is injective in a neighbourhood of $K$ and
\[
\sup_{z\in K}|f\circ \phi_n(z)-M| < \varepsilon.
\]
This implies that $\phi_n(K)$ has connected complement and that $\phi_n(K)\cap L=\emptyset$. Hence $\phi_n(K) \cup L$ is $\Omega$-convex.
\end{proof}

\begin{rem}\label{rem-two-holes-proof-last-assertion-thm}{\rm We observe that if $G$ has at least two holes and $\Omega$ is not infinitely connected, then for any compact set $K\in \MM(G)$ and any compact set $L\in \MM(\Omega)$, there cannot exist $n\in \N$ such that $\phi_n(K)$ is $\Omega$-convex and $\phi_n(K)\cap L=\emptyset$. Indeed, as in the proof of the last assertion of Theorem \ref{thm-nec-inj-one-infinity-holes}, this would imply that $\Omega$ has at least $q+p-1>q$ holes, which is impossible.
}
\end{rem}

The previous proof makes it appear that if there exists a $(\Phi,G,\Omega)$-universal function, then for any $N\in \N$, any compact set $K\in \MM(G)$ and any compact set $L\in \MM(\Omega)$ (hence also for any arbitrary compact set $L$ in $\Omega$, by Lemma \ref{exhaustion-convenient}), there exists an integer $n\geq N$ such that $\phi_n(K)$ is $\Omega$-convex and $\phi_n(K)\cap L=\emptyset$. In the particular case where $G$ is simply connected, we used the fact that these properties automatically imply that $\phi_n(K) \cup L$ is $\Omega$-convex. Thus, in this case, there exist $(\Phi,G,\Omega)$-universal functions if and only if for any $N\in \N$, any compact set $K\in \MM(G)$ and any compact set $L\in \MM(\Omega)$, there exists an integer $n\geq N$ such that $\phi_n(K)$ is $\Omega$-convex and $\phi_n(K)\cap L=\emptyset$.

\medskip

The following corollary shows that this equivalence holds not only if $G$ is simply connected.

\begin{coro}\label{coro-iff-simply-two-holes}
We assume that $\Phi =(\phi_n)_n$ is eventually injective and we make one of the following assumptions:
\begin{enumerate}
	\item $G$ is simply connected;
	\item $G$ has at least two holes.
\end{enumerate}
Then there exists a $(\Phi,G,\Omega)$-universal function if and only if for every compact set $K\in \MM(G)$, every compact set $L\in\Omega$, and every $N\in \N$, there exists $n\geq N$ such that $\phi_n(K)\cap L = \emptyset$ and $\phi_n(K)$ is $\Omega$-convex.
\end{coro}

\begin{proof}The "if" part when $G$ is simply connected and the "only if" part in both cases (1) and (2) have already been explained before the statement. Let us then focus on the "if" part when $G$ has at least two holes, and let us fix a compact set $K\in \MM(G)$, a compact set $L\in \MM(\Omega)$, and $N\in \N$. By Theorem \ref{thm-nec-inj-one-infinity-holes}, it is enough to show that if $\phi_n(K)\cap L = \emptyset$ and $\phi_n(K)$ is $\Omega$-convex for some $n\geq N$, then $\phi_n(K)\cup L$ is $\Omega$-convex. By Remark \ref{rem-two-holes-proof-last-assertion-thm}, we may and shall assume that $\Omega$ is infinitely connected (and in particular has more than two holes). We can proceed exactly as in the proof of \cite[Lemma 3.13]{G-Erdmann-Mortini}: consider an exhaustion $(K_l)_l$ of $G$ by compact sets in $\MM(G)$, with $K_1=K$, and an exhaustion $(L_l)_l$ of $\Omega$ by compact sets in $\MM(\Omega)$, with $L_1=L$. We use the hypothesis to get an increasing sequence $(n_l)_l\subset \N$ such that for any $l\in \N$, $\phi_{n_l}$ is injective on a neighbourhood of $K_l$, $\phi_{n_l}(K_l)$ is $\Omega$-convex and $\phi_{n_l}(K_l)\cap L_l =\emptyset$ (which implies $\phi_{n_l}(K)\cap L=\emptyset$); then we reproduce the same geometric argument as in the proof of \cite[Lemma 3.13]{G-Erdmann-Mortini}, with $\phi_{n_l}(K) \cup K$ replaced by $\phi_{n_l}(K) \cup L$, to get that $\phi_{n_l}(K)\cup L$ is $\Omega$-convex (here we use the assumption that $K$ and $L$ have at least two holes). The details are left to the reader.
\end{proof}



\medskip

Let us come back to the statement of Theorem \ref{thm-nec-inj-one-infinity-holes}. Even if it provides a necessary and sufficient condition for the existence of $(\Phi,G,\Omega)$-universal functions, a detail may look unpleasant: the fact this condition is stated for compact sets ($K$ and $L$) in $\MM(G)$ and $\MM(\Omega)$. The purpose of the next two lemmata is to make it understand that the condition is also necessary for compact subsets with less regularity.

\begin{lemme}\label{Lemma3.11}Let $K$ and $K'$ be two compact subsets of $G$, with $K\subset K'$. Let also $\phi$ be a holomorphic mapping from $G$ to $\Omega$ that is injective in a neighbourhood of $K'$. If $K$ is $G$-convex and $\phi(K')$ is $\Omega$-convex, then $\phi(K)$ is $\Omega$-convex.
\end{lemme}

\begin{proof}It is very similar to that of \cite[Lemma 3.11]{G-Erdmann-Mortini}, so is left to the reader.
\end{proof}

\begin{lemme}\label{lemme-tout-compact}
Let $K$ and $L$ be two compact sets, $G$-convex and $\Omega$-convex respectively. Let $\phi$ be a holomorphic mapping from $G$ to $\Omega$. We assume that $\Omega$ is not simply connected and that there exist two compact sets $K'\supset K$ in $\MM(G)$ and $L'\supset L$ in $\MM(\Omega)$ such that $\phi$ is injective in a neighbourhood of $K'$, $\phi(K')\cap L'=\emptyset$ and $\phi(K')\cup L'$ is $\Omega$-convex. Then $\phi$ is injective in a neighbourhood of $K$, $\phi(K)\cap L=\emptyset$ and $\phi(K)\cup L$ is $\Omega$-convex. 
\end{lemme}

\begin{proof}Let $K,K',L,L'$ be as in the statement of the lemma. Notice that $L'$ has at least one (non-empty) hole, since $L'\in \MM(\Omega)$ and $\Omega$ is not simply connected. Let us first show that $\phi(K')$ is $\Omega$-convex. If not, it has a hole $U$ contained in $\Omega$. Since $L'$ is connected, has a hole and is $\Omega$-convex, $L'$ does not intersect $U$. So $U$ is also a hole of $\phi(K')\cup L'$, a contradiction. By Lemma \ref{Lemma3.11}, it implies that $\phi(K)$ is $\Omega$-connected.

Now, by assumption, it is obvious that $\phi$ is injective in a neighbourhood of $K$ and that $\phi(K)\cap L=\emptyset$. It remains to check that $\phi(K)\cup L$ is $\Omega$-convex. By Lemma \ref{lemme-one-infinity-holes}, we are in one of the three following situations:
\begin{enumerate}
	\item $\phi(K')$ and $L'$ are in the unbounded connected component of the complement of the other, and so are $\phi(K)$ and $L$. Since the latter sets are both $\Omega$-convex respectively, $\phi(K)\cup L$ is $\Omega$-convex.
	\item $\phi(K')$ lies in a hole $O'$ of $L'$ and $O'\setminus \widehat{\phi(K')}$ contains a hole $O$ of $\Omega$. If $\phi(K)$ and $L$ are in the unbounded connected component of the complement of the other, then we are done. If not, then since $\phi(K)\subset \phi(K')$ and $L\subset L'$, $\phi(K)$ lies in a hole $O''$ of $L$ such that $O''\setminus \widehat{\phi(K)}$ contains $O'\setminus \widehat{\phi(K')}$, and thus contains $O$. It follows that any hole of $\phi(K)\cup L$ contains a hole of $\Omega$, so $\phi(K)\cup L$ is $\Omega$-convex.
	\item $L'$ lies in a hole of $\phi(K')$. The proof is very similar to that of (2) and is omitted.
\end{enumerate}
\end{proof}

The following statement, which is a direct consequence of Theorem \ref{thm-nec-inj-one-infinity-holes}, Corollary \ref{coro-iff-simply-two-holes} and the last two lemmata, summarizes the main result of this section, in the case $\Phi=(\phi_n)_n$ is eventually injective.

\begin{thm}\label{thm-charac-injective}
Let us assume that $\Phi =(\phi_n)_n$ is eventually injective. Then
\begin{enumerate}
	\item There exists a $(\Phi,G,\Omega)$-universal function if and only if for every $G$-convex compact set $K$, every $\Omega$-convex compact set $L$, and every $N\in \N$, there exists $n\geq N$ such that $\phi_n(K)\cap L = \emptyset$ and $\phi_n(K)\cup L$ is $\Omega$-convex.
	\item If $G$ is simply connected or has at least two holes, there exists a $(\Phi,G,\Omega)$-universal function if and only if for every $G$-convex compact set $K$, every compact set $L\subset \Omega$, and every $N\in \N$, there exists $n\geq N$ such that $\phi_n(K)\cap L = \emptyset$ and $\phi_n(K)$ is $\Omega$-convex.
	\item If $G$ is not simply connected and $\Omega$ is not infinitely connected, there do not exist $(\Phi,G,\Omega)$-universal functions.
\end{enumerate}
\end{thm}

\medskip

In the next example, we exhibit finitely connected domains $G$, that are not simply connected, and infinitely connected domains $\Omega$ for which $(\Phi,G,\Omega)$-universal functions may exist, with $\Phi$ injective. It also shows that the necessary and sufficient condition of Corollary \ref{coro-iff-simply-two-holes} is not sufficient any more if $G$ is doubly connected. Note that, by Remark \ref{rem-two-holes-proof-last-assertion-thm}, the assumption of Corollary \ref{coro-iff-simply-two-holes} is never satisfied if $G$ has at least two holes and $\Omega$ is not infinitely connected.

\begin{ex}\label{example-g-one-hole-omega-inf-conn}{\rm 
		
		(1) Let $G=\{z\in \C:\,2<|z|<4\}$. For $n\in \N$, we set $R_n=2^{-2n}$ and $r_n=2^{-2n-1}$. Let us also denote by $D_n$ the closed disc centred at $3.2^{-2n-2}$ with radius $2^{2n-2}$, and set
		\[
		\Omega= \C\setminus (\{0\}\cup \bigcup_{n\in \N}D_n).
		\]
		Then $\Omega$ is an infinitely connected domain. Moreover, since $r_{n-1}/R_n=2$, $G$ is conformally equivalent to each annulus $\{z\in \C:\,R_n<|z|<r_{n-1}\}$, $n\in \N$. Let us then denote by $\phi_n$ a conformal map from $G$ onto $\{z\in \C:\,R_n<|z|<r_{n-1}\}$, $n\in \N$. By construction, it is plain to check that if $K$ is a $G$-convex compact set and $L$ is an $\Omega$-convex compact set, then there exists $n\in \N$ such that each connected component of $\C\setminus (L\cup \phi_n(K))$ contains a hole of $\Omega$. So $L\cup \phi_n(K)$ is an $\Omega$-convex compact set and we can apply Theorem \ref{thm-charac-injective} to infer that there exist $(\Phi,G,\Omega)$-universal functions.
		
		(2) Let us consider $\Omega= \C \setminus (\{0\}\cup \bigcup_{n\in \N}B_n)$ where the $B_n$'s are any open disjoint discs contained in $\C \setminus \D$. Then, by the maximum modulus principle, we can check that with $G$ and $\Phi$ as in (1), there cannot exist $(\Phi,G,\Omega)$-universal functions (see Remark \ref{remark-cannot} for a similar argument). However, for every $G$-convex compact set $K$, every $\Omega$-convex compact set $L$, and every $N\in \N$, there exists $n\geq N$ such that $\phi_n(K)$ is $\Omega$-convex and $\phi_n(K)\cap L=\emptyset$.
		
		(3) Let $G$ be a bounded domain with $p$ holes, $0\leq p \leq \infty$. The first example suggests an \textit{ad-hoc} construction of a domain $\Omega$ and of an injective sequence $\Phi$ from $G$ to $\Omega$ that is universal. Denote by $\phi$ an arbitrary injective entire function such that $0\not\in \widehat{\overline{\phi(G)}}$ (for e.g., $\phi$ can be chosen as a translation). Let also $(a_l)_l \subset \C$ be a sequence tending to $0$. Since $\phi(G)$ is bounded, we can define by induction an increasing sequence $(l_n)_n$, with $l_0=0$, of positive integers and a decreasing sequence $(R_n)_n$ of positive numbers, with $R_n\to 0$, such that if we set $\phi_n=a_{l_n}\phi$, then
		\[
		\overline{\phi_{n+1}(G)}\subset A(0,R_{n+1},R_n),
		\]
		where $A(0,R_{n+1},R_n):=\{z\in \C: R_{n+1}< |z| < R_n\}$, $n\in \N$. Let $\mathcal{O}$ denote the set consisting of all the holes of $\cup_{n\in \N}\overline{\phi_n(G)}$. Then, for any $O\in \mathcal{O}$, let $D_O$ denote a non-empty closed disc contained in $O$ and consider the set
		\[
		\Omega:= \C \setminus \left(\{0\}\cup \bigcup _{O\in \mathcal{O}} D_O\right).
		\]
		By construction, it is readily checked that $\Phi:=(\phi_n)_n$ is a sequence of injective holomorphic mappings from $G$ to $\Omega$ such that if $K$ is a $G$-convex compact set and $L$ is an $\Omega$-convex compact set, then there exists $n\in \N$ such that each connected component of $\C\setminus (L\cup \phi_n(K))$ contains a hole of $\Omega$. Thus, by Theorem \ref{thm-charac-injective}, the sequence $\Phi$ is universal (with respect to $G$ and $\Omega$).
		
	}
\end{ex}

The last example implies the following proposition.

\begin{prop}Let $G$ be a bounded domain. There exist a domain $\Omega$, a number $a\in \C$ and an increasing sequence $(k_n)_n$ of integers such that the sequence $(\phi_n)_n$, defined by $\phi_n(z):=(z+a)/k_n$, is a universal sequence of injective holomorphic mappings from $G$ to $\Omega$.
\end{prop}

\medskip

If we do not impose eventual injectivity, it is not difficult to find pathological examples of sequences $\Phi=(\phi_n)_n$ acting from $G$ to $\Omega$, possibly both finitely connected and not simply connected, for which there exist $(\Phi,G,\Omega)$-universal functions. This can be seen from \cite[Proposition 3.6]{G-Erdmann-Mortini}. However, the same proof as that of (a)$\Rightarrow$(b)$\Rightarrow$(c) of \cite[Theorem 3.2]{G-Erdmann-Mortini} shows that for any sequence $\Phi$, the existence of $(\Phi,G,\Omega)$-universal functions always implies the existence of an increasing sequence $(n_j)_j$ such that for any compact set $K\subset G$ and any compact set $L\subset \Omega$, there exists $J\in \N$ such that $\Phi_{n_j}$ is injective on a neighbourhood of $K$ and $\Phi_{n_j}(K)\cap L=\emptyset$ for any $j\geq J$. Since the image of compact set with connected complement by a function which is injective and holomorphic on a neighbourhood of it still has connected complement, we deduce from Theorem \ref{thm-charac-injective} (2) the following statement.

\begin{thm}\label{thm-simply-connected-gl-GM}
If $G$ is simply connected, then there exists a $(\Phi,G,\Omega)$-universal function if and only if for any compact set $K$ in $G$ and any compact set $L$ in $\Omega$, there exists $n\in \N$ such that $\phi_n(K)\cap L=\emptyset$ and $\phi_n$ is injective on a neighbourhood of $K$.
\end{thm}

%

\section{The case where $G$ is a closed set with empty interior}\label{section-closed-G}

Let us introduce some notations that we will use in the rest of the paper.

\begin{itemize}
	\item We denote by $\T$ the unit circle $\{z\in \C:\, |z|=1\}$.
	\item If $K$ is a compact set in $\C$, we denote by $\CC(K)$ the Banach space of continuous functions endowed with the supremum norm.
	\item If $f\in H(\D)$ and $r\geq 0$, we denote by $f_r$ the function defined by $z\mapsto f(rz)$, $z\in D(0,1/r)$.
\end{itemize}

In the whole section, $\Omega$ is a domain and $G\subset \C$ is a closed set with empty interior. Let $\Phi=(\phi_n)_n$ be a sequence of continuous mappings from $G$ to $\Omega$. In this paragraph, we focus on the following definition. Let $\MM$ be a family of compact subsets of $G$.

\begin{defi}\label{defi-G-closed}
We say that $f\in H(\Omega)$ is $(\Phi,\MM)$-\textit{universal} if, for every compact set $K\in \MM$, the set $\{f\circ \phi_n:\,n\in \N\}$ is dense in $\CC(K)$.
\end{defi}

A typical example of holomorphic function universal in the sense of this definition is that of \emph{Abel universal functions}.

\begin{ex}[Abel universal functions]\label{Abel}{\rm Let $\Omega=\D$, $G=\T$, $\Phi=(\phi_n)_n$ with $\phi_n(z)=r_nz$, where $(r_n)_n$ is an increasing sequence in $[0,1)$, convergent to $1$. Let also $\MM$ be the set of all compact subsets of $\T$, different from $\T$. We call \emph{Abel universal functions} any $(\Phi,\MM)$-universal function.
	
The existence of Abel universal series was pointed out in \cite{Charp}. These are functions with a wild radial boundary behaviour. The terminology \emph{Abel universal functions} refers to the fact that a function $f\in H(\D)$ is said to be Abel summable at $\zeta \in \T$ if it has radial limit at $\zeta$.
}
\end{ex}

We will first state several sufficient or necessary conditions for the existence of $(\Phi,\MM)$-universal functions. Examples will be given at the end of the section. In particular, we will focus on applications to wild boundary behaviour of holomorphic functions.

\medskip 

We have the following Birkhoff's type criterion for $(\Phi,\MM)$-universality when $\MM$ is countable. The proof, similar to that of Fact \ref{fact-Birk-regular}, is also left to the reader.

\begin{fact}\label{Birk-G-empty}
Assume that $\MM$ is a countable family of compact subsets of $G$. 
Then there exists a $(\Phi,\MM)$-universal function if and only if for any $\varepsilon>0$, any compact set $K$ in $\MM$, any compact set $L$ of $\Omega$, any function $h\in \CC(K)$ and any function $g\in H(\Omega)$, there exist $f\in H(\Omega)$ and $n\in \N$ such that
\[
\sup_{z\in K}|f\circ \phi_n(z)-h(z)|\leq \varepsilon \quad \text{and} \quad \sup_{z\in L}|f(z)-g(z)|\leq \varepsilon.
\]
The set of $(\Phi,G)$-universal functions is either empty, or a dense $G_{\delta}$-subset of $H(\Omega)$.
\end{fact}

\medskip

Thanks to this criterion, we can deduce a sufficient condition for the existence of $(\Phi,\MM)$-universal functions.

\begin{prop}\label{thm-general}
Assume that $\MM$ is a countable family of compact subsets of $G$. There exists a $(\Phi,\MM)$-universal function in $H(\Omega)$ whenever for any compact set $K\in \MM$ and any compact set $L$ of $\Omega$, there exists $n\in \N$ such that $\phi_n$ is injective on $K$, $\phi_n(K)\cap L=\emptyset$ and $\phi_n(K)\cup L$ is $\Omega$-convex.
\end{prop}

\begin{proof}
In order to apply Fact \ref{Birk-G-empty}, let us fix $K\in \MM(G)$, $L$ a compact subset of $\Omega$, $g\in H(\Omega)$, $h\in \CC(K)$ and $\varepsilon>0$. By assumption, there exists $n\in \N$ such that $\phi_n$ is injective on $K$, $L\cap \phi_n(K)=\emptyset$ and $L\cup \phi_n(K)$ is $\Omega$-convex. Since $K$ has empty interior and $\phi_n$ is continuous on $K$, $\phi_n(K)$ has empty interior as well and, by Mergelyan's theorem, we can find a function $f\in H(\Omega)$ such that
	\[
	\sup_{L\cup \phi_{n}(K)}\left|f(z)- l(z)\right|\leq \varepsilon/2,
	\]
	where
	\[
	l(z)=\left\{\begin{array}{ll}g(z) & \text{for }z\in L\\
	h\circ \phi_n^{-1}(z) & \text{for }z\in \phi_{n}(K).\end{array}\right.
	\]
	One checks that $f$ is $\varepsilon$-close to $g$ on $L$ and that $f\circ \phi_n$ is $\varepsilon$-close to $h$ on $K$, and concludes by Fact \ref{Birk-G-empty}.
\end{proof}

\begin{rem}\label{rem-residual}{\rm Under the assumptions of Proposition \ref{thm-general}, the above sufficient condition ensures that the set of $(\Phi,\MM)$-universal functions is a dense $G_{\delta}$-subset of $H(\Omega)$. This will be used in Corollary \ref{coro-Jordan-union}.}
\end{rem}

Investigating when one can apply Proposition \ref{thm-general} requires to understand when $L\cup \phi_n(K)$ is $\Omega$-convex, for any given $\Omega$-convex compact set $L$ and any $K\in \MM$. Comparing to the situation of Section \ref{section-domain} where one may use the hole invariance by injective holomorphic mappings, here the problem is certainly too general to have a complete solution. Still, it becomes yet more realistic if we impose some conditions on $G$, or on $\Phi$. For instance, $L\cup \phi_n(K)$ is $\Omega$-convex whenever $\phi_n(K)$ has connected complement, $L$ is $\Omega$-convex and $L\cap \phi_n(K)$ is empty. Under the only assumption that $\phi_n$ is injective and continuous, it is not automatic that $\phi_n(K)$ has connected complement. Yet it holds true if $K$ is regular.

In the next corollary, we use the following terminology: a family $\FF$ of subsets of $\C$ is called uniformly separated if there exists $\delta >0$ such that for any $E,F\in \FF$, $\text{dist}(E,F)\geq \delta$, where $\text{dist}(E,F):=\inf\{|z-w|:\,z\in E,\,w\in F\}$.

\begin{coro}\label{coro-Jordan-union}
We assume that $G$ is a countable union of uniformly separated Jordan arcs or closed Jordan curves and that
$\MM$ is the set of all compact subsets of $G$ with connected complement. There exists a $(\Phi,\MM)$-universal function in $H(\Omega)$ whenever for any compact subset $L$ of $\Omega$ and any compact set $K\in \MM$ there exists $n\in \N$ such that $\phi_n$ is injective on $K$ and $\phi_n(K)\cap L=\emptyset$.
\end{coro}

The proof of this corollary is a consequence of Proposition \ref{thm-general}, up to two lemmata. The first one asserts that if $K$ a Jordan arc and $\phi$ is continuous and injective on $K$, then $\phi(K)$ has connected complement as well:

\begin{lemme}[Lemma 1.8 in \cite{Kolev}, for example]\label{geo-lemma1}If $\phi:[0,1]\to \C$ is continuous and injective then $\phi([0,1])$ has connected complement.
\end{lemme}

The second one states as follows.

\begin{lemme}\label{lem-Jordan-curve-seq-comp}Let $G$ be a finite union of pairwise disjoint Jordan arcs or Jordan curves. There exists a sequence $(K_n)_n$ of compact subsets of $G$, each of them with connected complement, such that for any compact set $K\subset G$, with connected complement, there exists $n\in \N$ such that $K\subset K_n$.
\end{lemme}

\begin{proof}Let us write $G=\tilde{G}\cup \bigcup _{k=1}^p G_k$, where $\tilde{G}$ is a finite union of Jordan arcs and where each $G_k$ is a Jordan curve, with $G_k\cap \tilde{G}=\emptyset$ and $G_k\cap G_{k'}=\emptyset$, $1\leq k\neq k'\leq p$. Let $(z^{(1)}_n,\ldots,z^{(p)}_n)_n\subset G$ be a dense sequence in $G_1\times \ldots \times G_p$, and let $(\eta_n)_n$ be a sequence in $(0,1)$, tending to $0$. For each $1\leq k\leq p$, let $\gamma_k:[0,1]\to G_k$ be a continuous mapping, injective on $[0,1)$, with $\gamma_k(0)=\gamma_k(1)$, and let $r^k_n:=\gamma_k^{-1}(z^{(k)}_n)$. For any $n\in \N$ and any $1\leq k\leq p$, we define the set
\[
\Gamma^k_n:=\gamma^k\left([0,r^k_n-\eta_n]\cup [r^k_n - \eta_n,1]\right).
\]
By definition, for any $n\in \N$, the set $K_n:=\tilde{G}\cup\bigcup _{k=1}^p\Gamma^k_n$ is a finite disjoint union of Jordan arcs, hence it is compact and has connected complement by Lemma \ref{geo-lemma1}. Now it is clear that any compact subset of $G$, with connected complement, is contained in some $K_n$, $n\in \N$.
\end{proof}


\begin{proof}[Proof of Corollary \ref{coro-Jordan-union}]Let us write
\[
G=\bigcup_{k\geq 1}g_k\quad \text{and}\quad G_n=\bigcup_{k=1}^ng_k
\]
where each $g_k$ is a Jordan arc or a Jordan curve, with $g_k\cap g_{k'}=\emptyset $ for any $k\neq k'$. First, if $G$ is replaced with any $G_n$, $n\geq 1$, then, by Lemma \ref{lem-Jordan-curve-seq-comp}, we can assume that $\MM$ is countable. Thus it directly follows from Proposition \ref{thm-general} (more precisely Remark \ref{rem-residual}) that the set $\UU_n$ of $(\Phi|G_n,\MM_n)$-universal functions is a dense $G_{\delta}$-subset of $H(\Omega)$, for any $n\geq 1$, where $\MM_n$ denotes the set of all compact sets in $\MM$ which are contained in $G_n$. Now, since $G$ is assumed to be a uniformly separated union of the $g_k$'s, any set of $\MM$ is an element of $\MM_n$ for some $n$, $n\geq 1$. Therefore, it is easily seen that the set $\UU$ of all $(\Phi,\MM)$-universal functions coincides with the intersection of all the $\UU_n$, and hence is also a dense $G_{\delta}$-subset of $H(\Omega)$ by Baire's theorem.
\end{proof}

\medskip

It is natural to wonder whether the sufficient condition for the existence of $(\Phi,\MM)$-universal functions given by Proposition \ref{thm-general} or Corollary \ref{coro-Jordan-union} is also necessary. The next proposition tells us that it is the case if $\Phi$ is eventually injective. More generally, $\Phi$ cannot be "too much non-injective".

\begin{prop}\label{prop-nec-G-empty-interior}
Assume that $\MM$ is a countable family of compact subsets of $G$. If there exists a $(\Phi,\MM)$-universal function in $H(\Omega)$, then for any compact set $L\in \Omega$, any compact set $K\in \MM$, and any disjoint compact subsets $I_1,I_2\in \MM$ with $K\cup I_1 \cup I_2 \in \MM$, there exist infinitely many $n\in \N$ such that the following two conditions hold:
	\begin{enumerate}
		\item $\phi_n(K)\cap L=\emptyset$;
		\item $\phi_n(I_1)\cap \phi_n(I_2)=\emptyset$;
	\end{enumerate}
\end{prop}

\begin{proof}Let us fix $L$, $K$ and $I_1,I_2$ as in the statement. 
	We set $\delta=\text{dist}(I_1,I_2)>0$. If $f\in H(\Omega)$ is a $(\Phi,\MM)$-universal function, then for any $j\geq 1$, there exists $n_j\in \N$ such that
	\begin{equation}\label{eqnec1}
	\left|f\circ \phi_{n_j} (z) - (j+z)\right| < \frac{\delta}{2},\quad z\in K\cup I_1 \cup I_2.
	\end{equation}
	Thus for any $z\in K$, $|f\circ \phi_{n_j}(z)|> j+|z|-\delta$ and we can find $J\in \N$ large enough so that $\inf_{K} |f\circ \phi_{n_j}|\geq \sup _L |f|$ for any $j\geq J$, whence $\phi_{n_j}(K)\cap L =\emptyset$ for any $j\geq J$.
	
	For the second condition, let us assume by contradiction that there exists $J'\in \N$, $J'\geq J$, such that $\phi_{n_j}(I_1)\cap \phi_{n_j}(I_2)\neq \emptyset$ for any $j\geq J'$. Then there exists $(\zeta _j)_j\subset I_1$ and $(\xi_j)_j\in I_2$ such that $\phi_{n_j}(\zeta_j)=\phi_{n_j}(\xi_j)$. Now, by \eqref{eqnec1} we get that for any $j\geq J'$ and any $z\in I_1 \cup I_2$,
	\[
	\left|f\circ \phi_{n_j} (\zeta_j) - (j+\zeta_j)\right| < \frac{\delta}{2}\quad \text{and} \quad \left|f\circ \phi_{n_j} (\xi_j) - (j+\xi_j)\right| < \frac{\delta}{2},
	\]
	hence $\left|\zeta_j - \xi_j\right| < \delta$. This contradicts the definition of $\delta$.
\end{proof}

\begin{ex}\label{rem-nec-compo}{\rm Let $\Omega=\D$, $G=\T$ and let $\MM$ be the set of all compact subsets of $\T$, different from $\T$. The assumption (2) in Proposition \ref{prop-nec-G-empty-interior} is not satisfied if there exist $\zeta,\xi\in \T$ with $\zeta\neq \xi$, and two sequences $(\zeta_n)\to \zeta$ and $(\xi_n)\to \xi$ in $\T$ such that $\phi_n(\zeta_n)=\phi_n(\xi_n)$ for any $n$.

For example, let $(r_n)_n$ be any sequence in $[0,1)$ tending to $1$. If $\Psi=(\psi_n)_n$ is defined for $\zeta \in \T$ by $\psi_n(\zeta)=r_n\zeta^k$, for some fixed $k\geq 2$, or by $\psi_n(\zeta)=r_n\zeta^n$, then there is no $(\Psi,\MM)$-universal functions.

This can be compared with the existence of Abel universal functions (Example \ref{Abel}), that are $(\Phi,\MM)$-universal functions with $\Phi=(\phi_n)_n$ defined by $\phi_n(\zeta)=r_n\zeta$, $\zeta\in \T$.}
\end{ex}

To conclude this section, we shall exhibit sequences $\Phi$ of continuous mappings from $\T$ to $\D$ for which $(\Phi,\MM)$-universal functions exist, even if no subsequence of $\Phi$ is eventually injective. This indicates that the previous proposition is not far from being optimal, and this highlights the contrast with the setting considered in the previous section (see the comment before Theorem \ref{thm-simply-connected-gl-GM}). For simplicity, we will assume that $G$ is the unit circle $\T$ but one could also consider that $G$ is a union of uniformly separated Jordan arcs or Jordan curves.

We recall the following notation. For $\phi:\T\to \C$ continuous and $K\subset \T$, we denote by $\widehat{\phi(K)}$ the union of $\phi(K)$ and all bounded connected components of $\C\setminus \phi(K)$. Observe that since the set $\phi(K)$ is compact, $\C\setminus \phi(K)$ has only one unbounded connected component and that $\widehat{\phi(K)}$ is a compact set with connected complement. Without possible confusion, we will use the following notations, only in the next statement, in its proof, and in Example \ref{ex-4-12}: if $0\leq \alpha\leq \beta \leq 2\pi$, we will denote by $[e^{i\alpha},e^{i\beta}]$ the set $\{e^{i\theta}:\,\alpha\leq \theta \leq \beta\}$, by $]e^{i\alpha},e^{i\beta}]$ the set $\{e^{i\theta}:\,\alpha< \theta \leq \beta\}$, \emph{et caetera}.

\begin{prop}\label{suff-plus}Let $G=\T$ and let $\MM$ be the set of all compact subsets of $\T$, different from $\T$. We assume that for any compact set $L\subset \Omega$, any proper closed arc $K=[e^{i\alpha},e^{i\beta}]\subset \T$ and any $\delta >0$, there exists $n\in \N$ such that the following two conditions hold:
	\begin{enumerate}
		\item $\phi_n(K)\cap L=\emptyset$;
		\item there exist $\alpha\leq \delta_1 \leq \delta_2 < \delta_3 \leq \delta_4 < \ldots < \delta_{2l-1} \leq \delta _{2l} \leq \beta$, such that if we set $I_m=\left[e^{i\delta_{2m-1}},e^{i\delta_{2m}}\right]$, then
		\begin{enumerate}
			\item $\text{length}(I_m)\leq \delta$, $m=1,\ldots,l$;
			\item $\widehat{\phi_n(I_m)}\cap \phi_n(K\setminus I_m)=\emptyset$, $m=1,\ldots,l$;
			\item $\phi_n|U$ is injective where $U=K\setminus \cup_mI_m$.
		\end{enumerate}
	\end{enumerate}
Then there exist $(\Phi,\MM)$-universal functions.
\end{prop}

For the proof of this proposition, we will make use of the following easy geometric lemma. It is a straightforward consequence of Kallin's lemma about polynomial convexity of the union of polynomially convex subsets of $\C^n$, $n\geq 1$ (see, for e.g., \cite[Theorem 1]{Paepe}). Yet, since the case $n=1$ is quite simple, we outline a proof here.

\begin{lemme}\label{geo-lemma2}Let $K_1$ and $K_2$ be two compact subsets such that $K_1\cap K_2$ contains at most $1$ element. If $K_1$ and $K_2$ both have connected complement, then $K_1\cup K_2$ also has connected complement.
\end{lemme}

\begin{proof}
We only prove the case where $K_1\cup K_2 = \{w_0\}$, the case where $K_1$ and $K_2$ are disjoint being contained in Lemma \ref{lem-union-omega-convex}. Let $z_1,z_2 \in \C\setminus K_1\cup K_2$. By compactness of $K_1 \cup K_2$, we can find two closed Jordan domains $V_1$ and $V_2$ such that $K_1\setminus \{w\}\subset \overset{\circ}{V_1}$, $K_2\setminus \{w\}\subset \overset{\circ}{V_2}$, $V_1\cap V_2 = \{w_0\}$ and $z_1,z_2 \in \C\setminus V_1\cup V_2$. Since $V_2$ has connected complement and $V_1\setminus \{w_0\}\subset \C \setminus V_2$, there exist two continuous paths $\gamma_1$ and $\gamma_2$ such that, for $i=1,2$, one of the extremity $\xi_{i,1}$ of $\gamma_i$ is $z_i$, the other $\xi_{i,2}$ is in $\partial V_1 \setminus \{w_0\}$, and $\gamma_i \setminus \{\xi_{i,2}\}\subset \C\setminus V_1\cup V_2$. Since $V_1$ is a Jordan domain, $\partial V_1 \setminus \{w\}$ is path connected and there exists a path $\gamma \in \partial V_1 \setminus \{w\}$ whose extremities are $\xi_{1,2}$ and $\xi_{2,2}$. The union $\gamma_1\cup \gamma_2 \cup \gamma$ is a continuous path, which by construction lies in $\C\setminus (K_1 \cup K_2)$. Thus the latter set is path connected.
\end{proof}

\begin{proof}[Proof of Proposition \ref{suff-plus}]By conformal invariance, we can assume that $\Omega=\D$. In order to apply Fact \ref{Birk-G-empty}, let us fix $\varepsilon>0$, a compact set $K\subset \T$ with connected complement, a compact set $L\in \Omega$, a function $h\in \CC(K)$ and a function $g\in H(\Omega)$. By Mergelyan's theorem, we may and shall assume $h$ is a polynomial and that $K$ is an arc of the form $[e^{i\alpha},e^{i\beta}]$. It is enough to prove that there exist $f\in H(\Omega)$ and $n\in \N$ such that
\[
\sup_{z\in K}\left|f\circ \phi_n(z) - h(z)\right|\leq \varepsilon\quad \text{and}\quad \sup_{z\in L}\left|f(z) - g(z)\right|\leq \varepsilon.
\]
Since $h$ is uniformly continuous on $K$, there exists $\delta>0$ such that for any $\zeta,\xi \in K$,
\begin{equation}\label{unifcont}
\left|\zeta -\xi\right|\leq 2\delta \quad \Longrightarrow \quad \left|h(\zeta) - h(\xi)\right|\leq \frac{\varepsilon}{2}.
\end{equation}
By assumption, there exist $n\in \N$ and closed arcs $I_1,\ldots,I_l \subset K$ as in the statement (in particular satisfying (1) and (2)). Let us first prove that the compact set
\[
\bigcup _{m=1}^l\widehat{\phi_n(I_m)}\cup \phi_n(U)\cup L
\]
has connected complement. By (1) and Lemma \ref{geo-lemma2}, it is enough to prove that
\[
\bigcup _{m=1}^l\widehat{\phi_n(I_m)}\cup \phi_n(U)
\]
has connected complement. Let $J_0=\left[e^{i\alpha},e^{i\delta_{1}}\right[$, $J_l=\left]e^{i\delta_{2l}},e^{i\beta}\right]$ and $J_m=\left]e^{i\delta_{2m}},e^{i\delta_{2m+1}}\right[$, $m=1,\ldots, l-1$. Now, the assumption (2) (c) and Lemma \ref{geo-lemma1} imply that $\overline{\phi_n(J_m)}=\phi_n(\overline{J_m})$, $m=0,\ldots ,l$, has connected complement. Moreover, the assumption (2) (b) implies that $\widehat{\phi_n(I_m)}$ and $\overline{\phi_n(J_{m-1})}$ do intersect at, at most, one point. Thus, by Lemma \ref{geo-lemma2}, for any $m=1,\ldots ,l$, the set
\[
\widehat{\phi_n(I_m)}\cup \overline{\phi_n(J_{m-1})}
\]
has connected complement. Furthermore, by assumptions (2) (b) and (2) (c), for any $m'\in \{1,\ldots,l-1\}$, the sets
\[
\left(\bigcup_{m=1}^{m'}\widehat{\phi_n(I_m)}\cup \overline{\phi_n(J_{m-1})}\right)\cap \left(\widehat{\phi_n(I_{m'+1})}\cup \overline{\phi_n(J_{m'})}\right)
\]
and
\[
\left(\bigcup_{m=1}^{l-1}\widehat{\phi_n(I_l)}\cup \overline{\phi_n(J_{l-1})}\right)\cap \overline{\phi_n(J_{l})},
\]
contain at most one element. Therefore Lemma \ref{geo-lemma2} and a finite induction allow us to conclude that
\[
\bigcup _{m=1}^l\widehat{\phi_n(I_m)}\cup \overline{\phi_n(U)}=\bigcup _{m=1}^l\widehat{\phi_n(I_m)}\cup \phi_n(U)
\]
has connected complement.

Let us now fix $\eta_m >0$, $m=1,\ldots, l$ such that for any $m=1,\ldots, l$
\begin{enumerate}[(a)]
		\item $\delta_{2m} + \eta _m < \delta _{2m+1}$;
		\item $\delta_{2m}-\delta_{2m-1} + \eta_m < 2\delta$.
	\end{enumerate}
	For any $m=1,\ldots,l$, let also $h_m$ be any continuous map such that 
	\[
	h_m\left(\left[e^{i\delta_{2m}},e^{i(\delta_{2m}+\eta_m)}\right]\right)=\left[e^{i\delta_{2m-1}},e^{i(\delta_{2m}+\eta_m)}\right],
	\]
	$h_m\left(e^{i\delta_{2m}}\right)=e^{i\delta_{2m-1}}$ and $h_m\left(e^{i(\delta_{2m}+\eta_m)}\right)=e^{i(\delta_{2m}+\eta_m)}$. Then we consider the function $\tilde{h}:\phi_n(K)\to \C$ defined by
	\[
	\tilde{h}(z)=\left\{\begin{array}{ll}
	g(z) & \text{if } z\in L\\
	h\left(\phi_n^{-1}(z)\right) & \text{if } z\in \phi_n(J_0)\\
	h\left(e^{i\delta_{2m-1}}\right) & \text{if } z\in \widehat{\phi_n(I_m)},\,m=1,\ldots,l\\
	h\left(h_m\circ \phi_n^{-1}(z)\right) & \text{if } z\in\phi_n\left(\left]e^{i\delta_{2m}},e^{i(\delta_{2m}+\eta_m)}\right]\right),\,m=1,\ldots,l\\
	h\left(\phi_n^{-1}(z)\right) & \text{if } z\in \phi_n\left(\left[e^{i(\delta_{2m}+\eta_m)},e^{i\delta_{2m+1}}\right[\right),\,m=1,\ldots,l-1\\
	h\left(\phi_n^{-1}(z)\right) & \text{if } z\in \phi_n(J_l).
	\end{array}\right.
	\]
	
	One can check that by construction $\tilde{h}$ is continuous on
	\[
	A:=\bigcup _{m=1}^l\widehat{\phi_n(I_m)}\cup \phi_n(U)\cup L,
	\]
	holomorphic in its interior, and we can thus apply Runge's theorem to get a function $f\in H(\Omega)$ such that
	\[
	\sup_{z\in A}\left|f(z) - \tilde{h}(z)\right|\leq \frac{\varepsilon}{2},
	\]
	Since $\phi_n(K) \subset A$, we get
	\[
	\sup_{\zeta\in K}\left|f(\phi_n(z)) - \tilde{h}(\phi_n(z))\right|\leq \frac{\varepsilon}{2}.
	\]
	By construction, it is clear that for any $\zeta\in J_0\cup J_l \cup \bigcup_{m=1}^{l-1}\left[e^{i(\delta_{2m}+\eta_m)},e^{i\delta_{2m+1}}\right]$, one has
	\[
	\left|f\circ\phi_n(z) - h(\zeta)\right| \leq \frac{\varepsilon}{2}.
	\]
	Moreover, for $\zeta \in I_m$, $m=1,\ldots,l$, since $\phi_n(\zeta) \in \widehat{\phi_n(I_m)}$ and because of \eqref{unifcont} and the choice of $\eta_m$, we have
	\[
	\left|f\circ \phi_n(\zeta) - h(\zeta)\right| \leq \left|f\circ \phi_n(\zeta) - h(e^{i\delta_{2m-1}})\right| + \left|h(e^{i\delta_{2m-1}}) - h(\zeta)\right| \leq \varepsilon.
	\]
	The same argument gives that for any $\zeta \in \left[e^{i\delta_{2m}},e^{i(\delta_{2m}+\eta_m)}\right]$, $m=1,\ldots,l$,
	\[
	\left|f\circ \phi_n(\zeta) - h(\zeta)\right| \leq \varepsilon,
	\]
	which concludes the proof.
\end{proof}

\medskip

In order to illustrate Proposition \ref{suff-plus}, let us give an example of a sequence $(\phi_n)_n$ of continuous mappings from $\T$ to $\D$, that is not eventually injective, and for which $(\Phi,\MM)$-universal functions exist (where $\MM$ is the set of all compact subsets of $\T$, different from $\T$).

\begin{ex}\label{ex-4-12}{\rm 
For any $n\in \N$, $n\geq 2$, we set $\psi_n(e^{i\vartheta})=e^{i\frac{n\vartheta-\pi}{n-1}}$, $\vartheta\in [\pi/n,2\pi - \pi/n]$, and define
\[
\phi_n(z):=\left\{\begin{array}{ll}
1-\frac{3}{2n}-\frac{\bar{z}^n}{2n} & \text{if }z\in [1,e^{i\pi/n}]\\
\left(1-\frac{1}{n}\right)\psi_n(z) & \text{if }z\in [e^{i\pi/n},e^{2i\pi-i\pi/n}]\\
1-\frac{3}{2n}-\frac{z^n}{2n} & \text{if }z\in [e^{2i\pi-i\pi/n},1].
\end{array}\right.
\]
The image of each $\phi_n$ is the union of two circles, tangent at $1-1/n$. It is clear that the sequence $(\phi_n)_n$ satisfies the assumptions of Proposition \ref{suff-plus}. Moreover, it is also clear that for any arc $K\subset \T$ with $1\in K$, only finitely many $\phi_n$ are injective on $K$, hence $(\phi_n)_n$ fails to be eventually injective.

What makes the example satisfy the assumptions of Proposition \ref{suff-plus} is the fact that on the preimage of the small circle of $\phi_n(\T)$ (i.e. on $[1,e^{i\pi/n}]$ and $[e^{2i\pi-i\pi/n},1]$)), the derivative of the map $z\mapsto \text{arg}(\phi_n(z))$ is large. In particular, it appears that the parametrization of $\phi_n(\T)$ by $\phi_n$ is of importance: one could find another sequence $(\psi_n)_n$, with $\psi_n(\T)=\phi_n(\T)$ for any $n\in \N$, but so that the sequence $(\psi_n)_n$ is not universal.

The sequence $(\phi_n)_n$ can easily be modified to get more sophisticated examples.
}
\end{ex}

\medskip

It may seem natural to seek for necessary and sufficient condition for the existence of $(\Phi,\MM)$-universal functions, even in the case where $\Omega=\D$, $G=\T$ and $\MM$ is the set of all compact subsets of $\T$, different from $\T$. Yet this problem seems (at least to us) too general to expect a solution.

\subsubsection*{Application to the boundary behaviour of holomorphic functions}

Corollary \ref{coro-Jordan-union} can be interpretated in terms of boundary behaviour of holomorphic functions in several manners. Let us illustrate this by the following examples.

\medskip

(1) It is well-known that for any $f \in H(\C^*)$ having an essential singularity at $0$, there exists a sequence $(z_n)_n \subset \C^*$, converging to $0$ such that the set $\{f(z_n):\,n\in \N\}$ is dense in $\C$. 
One can immediately deduce from Corollary \ref{coro-Jordan-union} that, given a sequence of complex numbers $(\lambda_l)_l$, with no accumulation points, and sequences of complex numbers $(z_n^l)_n$, $l=1,2,\ldots$, such that $z_n^l \to \lambda_l$ for any $l=1,2,\ldots$, there exists a dense $G_{\delta}$ set of functions $f\in H(\C\setminus \{\lambda_n:\,n=1,2,\ldots\})$ such that the sets $\{f(z_n^l):\,n\in \N\}$, $l=1,2,\ldots$, are dense in $\C$.
		
\medskip

(2) Let $\Omega=\C\setminus \D$, $G=\T$ and $\Phi=(\phi_n)_n$ from $\T$ to $\D$, defined by $\phi_n(z)=R_nz$, where $(R_n)_n$ is a sequence in $]1,+\infty[$ converging to $1$. Corollary \ref{coro-Jordan-union} gives us the existence of functions $f$ in $H(\Omega)$ such that for any compact sets $K\subset \T$, the set $\{f\circ \phi_n:\,n\in \N\}$ is dense in $\CC(K)$.

\medskip
		
More generally, let $\Omega=\C\setminus \cup_{l\in \N}D_l$ where $(D_l)_l$ is uniformly separated closed discs, and let $G=\cup _{l\in\N}\partial D_l$. Let also $(\phi_n)_n$ be a sequence of injective continuous mappings from $G$ to $\Omega$. Then there exist functions $f$ in $H(\Omega)$ such that for any compact set $K\subset G$, with connected complement, the set $\{f\circ \phi_n:\,n\in \N\}$ is dense in $\CC(K)$. Note that each $\phi_n$ can be defined independently on each $\partial D_l$. Moreover, the discs $D_l$ can be replaced by Jordan domains.
		
For instance, choosing the functions $\phi_n$ such that $\phi_n(z)\to z$ as $n\to \infty$, we obtain examples of functions holomorphic on domains, possibly infinitely connected, with very wild boundary behaviour.

\medskip
		
(3) Let $\Omega =\C\setminus [0,1]$. The results of this section yields the existence of a $G_{\delta}$-dense set of functions $f$ in $H(\Omega)$ such that the sequence of functions $(f(z+i/n))$ is dense in $\CC([0,1])$.

\medskip

(4) Let us fix a sequence $(r_n)_n$ in $(0,1)$, tending to $1$, and let $\Omega$ be a bounded finitely connected domain with $\CC^1$ boundary (i.e. each connected complement of $\partial \Omega$ is a $\CC^1$ Jordan curve, that is the image of $\T$ by a $\CC^1$-diffeomorphism). Let $\MM$ be the set of all compact subsets of $\partial \Omega$ with connected complement. For $z\in \partial \Omega$, let us denote by $\nu_z$ the inward unit normal vector to $\partial \Omega$ at $z$, and set $\phi_n(z)=z+(1-r_n)\nu_z$, $z\in \partial \Omega$. Note that $\phi_n(z)\to z$ as $n\to \infty$ for any $z\in \partial \Omega$. It is easily checked that $(\phi'_n)_n$ tends to the constant map $1$ uniformly on $\partial \Omega$ so that, for any $n$ large enough, $\Phi:=(\phi_n)_n$ is a sequence of injective continuous mappings from $\partial \Omega$ to $\Omega$. Then, for such $n$, $\phi_n(\partial \Omega)$ defines the boundary of a compact subset $K_n$ of $\Omega$ such that $(K_n)_n$ is an exhaustion of $\Omega$. Finally, it is clear that for any $K\in \MM$ and any compact set $L\in \Omega$, there exists $n\in \N$ such that $\phi_n(K)\cap L=\emptyset$.

Thus Corollary \ref{coro-Jordan-union} implies the existence of some $(\Phi,\MM)$-universal function. By construction, the boundary behaviour of such function is erratic near any compact subset of $\partial \Omega$.

Up to non-essential modifications, the previous can be extended to infinitely connected domain whose boundary is a uniformly separated family of $\CC^1$ Jordan curves.

\medskip

(5) Let now $\Omega$ be a Jordan domain and let $H:\partial \Omega \times [0,1) \to \Omega$ be continuous and such that $H(z,r)\to z$ as $r\to 1$, for any $z\in \partial \Omega$, and such that for any $r\in [0,1)$, the map $z\mapsto H(z,r)$ is injective on $\partial \Omega$. Note that such map exists for any Jordan domain $\Omega$ (for example, the map $H(z,r)=\omega (rz)$, where $\omega$ is a conformal map from $\D$ to $\Omega$). Let $(r_n)_n$ be an increasing sequence in $[0,1)$ tending to $1$ and, for $n\in \N$, denote by $\psi_n:\partial \Omega \to \Omega$ the continuous map defined by $\psi_n(z)=H(z,r_n)$. Clearly, $(\psi_n)_n$ satisfies the sufficient condition in Corollary \ref{coro-Jordan-union}, so we can derive the following extension of the notion of Abel universal functions to any Jordan domain.

\begin{coro}\label{coro-Abel-Jordan-Domain}Let $\Omega$ and $(\psi_n)_n$ be as above. There exists a function $f\in H(\Omega)$ such that for any compact set $K\subset \partial \Omega$ different from $\partial \Omega$, the set $\{f\circ \psi_n:\, n\in \N\}$ is dense in $\CC(K)$.
\end{coro}

%
%

\medskip

Let us finish by a remark.

\begin{rem}\label{remark-cannot}{\rm Let $\Omega=\C\setminus \D$, $G=\T$ and $\Phi=(\phi_n)_n$ from $\T$ to $\D$, defined by $\phi_n(z)=R_nz$, where $(R_n)_n$ is a sequence in $]1,+\infty[$ converging to $1$. We have seen in the above example (1) that there exist $(\Phi,\MM)$-universal functions, where $\MM$ is the set of all compact subset of $G$, different from $G$. There cannot exist functions $f\in H(\Omega)$ that are $(\Phi,\{G\})$-universal. Indeed, if $f$ were such a function, there would exist three closed Jordan curve $\Gamma_1$, $\Gamma_2$ and $\Gamma_3$ (respectively images of $\T$ by $\phi_{n_1}$, $\phi_{n_2}$ and $\phi_{n_3}$ with $n_1$, $n_2$ and $n_3$ large enough), such that the closed Jordan domain bounded by $\Gamma_2$ is contained in the interior of that bounded by $\Gamma_1$, and the closed Jordan domain bounded by $\Gamma_3$ is contained in the interior of that bounded by $\Gamma_2$, so that $f$ is close to $0$ on $\Gamma_1\cup \Gamma_3$ and as big as any fixed positive number on $\Gamma_2$, contradicting the maximum modulus principle applied to $f$ on the domain bounded by $\Gamma_1\cup \Gamma_3$.
		
The same kind of argument ensures that, if $\Omega=\D$, $G=\T$ and $\Phi=(\phi_n)_n$ from $\T$ to $\D$, defined by $\phi_n(z)=r_nz$, where $(r_n)_n$ is a sequence in $]0,1[$ converging to $1$, then there do not exist $(\Phi,\{G\})$-universal functions.}
\end{rem}

\end{document}